\documentclass[12pt,reqno,a4paper]{amsart}
\usepackage{amsmath,amssymb,amsthm}
\nonstopmode \numberwithin{equation}{section}
\newtheorem{theorem}{Theorem}[section]

\newtheorem{remark}{Remark}[section]

\newtheorem{lemma}{Lemma}[section]
\newtheorem{corollary}{Corollary}[section]
\newtheorem{example}{Example}[section]
\begin{document}
\date{}
\title{Geometric properties of Ces\`aro averaging operators }
\author{
Priyanka Sangal$^{\ast}$}
\thanks{$^{\ast}$Corresponding author}
\address{Department of Mathematics, IIT Roorkee}
\email{sangal.priyanka@gmail.com, priyadma@iitr.ac.in}

\author{A. Swaminathan}
\address{
Department of  Mathematics  \\
Indian Institute of Technology, Roorkee-247 667,
Uttarkhand,  India
}
\email{swamifma@iitr.ac.in, mathswami@gmail.com}

\bigskip

\maketitle
\begin{abstract}
In this paper, using positivity of trigonometric cosine and sine sums
whose coefficients are generalization of Vietoris numbers, we find the conditions on
the coefficient $\{a_k\}$ to characterize the geometric properties of the corresponding
analytic function $f(z)=z+\displaystyle\sum_{k=2}^{\infty} a_kz^k$ in the unit disc $\mathbb{D}$.
As an application we also find geometric properties of a generalized Ces\`aro type polynomials.
\end{abstract}



\pagestyle{myheadings}\markboth{Priyanka Sangal and A. Swaminathan}
{Geometric properties of Ces\`aro averaging operators}

\section{introduction}
Inequalities involving trigonometric sums arise naturally in various
problems of pure and applied mathematics. Inequalities that assure
nonnegativity or boundedness of partial sums of trigonometric series
are of particular interest and applications in various fields. For example,
the positivity of trigonometric polynomials are studied in geometric function theory
by Gluchoff and Hartman \cite{gluchoff-hartman-1998-AMS}
and Ruscheweyh and Salinas \cite{ruscheweyh-salinas-2004-stable-JMAA}.
For a detailed application in signal processing, we refer to the monograph of
Dumitrescu \cite{signal-processing-book}.
For other applications in this direction we refer to Dimitrov and Merlo
\cite{dimitrov-merlo-2002-cons-approx}, Fernandez-Duran \cite{Fernandez-2004},
Gasper \cite{gasper-1969-nonneg-sum-JMAA}.
The positive trigonometric polynomials played important role in the proof of Bieberbach
conjecture, see \cite{askey-gasper-biebarbach-conj-AMS-1986}.
For the applications of positive trigonometric polynomials in Fourier series,
Approximation theory, Function Theory and Number Theory, we refer to the work
of Dimitrov \cite{dimitrov-extremal-trig-pol-2002} and references therein. For the study of
extremal problems we refer to the dissertation of Revesz \cite{revesz-thesis}
wherein several applications are outlined.

The problem of finding the behaviour of the coefficients to validate the
positivity of trigonometric sum has been
dealt by many researchers. Among them the contributions of
Vietoris \cite{vietoris-1958} followed by Koumandos \cite{koumandos-2007-ext-viet-ramanujan}
are of interest to the present investigation. Precisely Vietoris \cite{vietoris-1958} gave sufficient
conditions on the coefficient of a general class of sine and
cosine sums that ensure their positivity in $(0,\pi)$.
For further details in this direction one can refer to
\cite{brown-dai-wang-2007-ext-viet-ramanujan,koumandos-2007-ext-viet-ramanujan,
sangal-swaminathan-positivity-alpha-beta} and the references therein.
An account of recent results available in this direction is given in
\cite{sangal-swaminathan-positivity-alpha-beta} and one of the
main result in \cite{sangal-swaminathan-positivity-alpha-beta} is the following.
\begin{theorem}\cite{sangal-swaminathan-positivity-alpha-beta}
\label{thm:new-alpha-beta-positive-sum}
Suppose that $\alpha\geq0,\beta\geq 0$, and $\lambda,\mu \geq 0$ such that
$\lambda+\mu\geq 1$ then for  $b_0=2,b_1=1$ and
$b_k=\frac{1}{(k+\alpha)^{\lambda}(k+\beta)^{\mu}}, k\geq2$,
we have,
\begin{align*}
\frac{b_0}{2}+\sum_{k=1}^nb_k \cos{k \theta} &>0\quad \mbox{and }
\quad \sum_{k=1}^nb_k \sin{k \theta} >0,
\end{align*}
for $0< \theta<\pi$ and $n\in\mathbb{N}$.
\end{theorem}
Using summation by parts the following corollary of Theorem
\ref{thm:new-alpha-beta-positive-sum} can be obtained.
\begin{corollary}\label{cor:new-alpha-beta-positive-sum}
For $\alpha\geq0,\beta\geq0$ and $\lambda\geq0,\mu\geq0$ such that
$\lambda+\mu\geq1$.
If there exists a sequence $\{a_k\}$ of positive real numbers such that,
\begin{align*}
(k+1+\alpha)^{\lambda}(k+1+\beta)^{\mu}a_{k+1} \leq (k+\alpha)^{\lambda}(k+\beta)^{\mu}
a_k\leq \cdots \leq (2+\alpha)^{\lambda}(2+\beta)^{\mu}a_2 \leq a_1 \leq \frac{a_0}{2}
\end{align*}
Then for $n\in\mathbb{N}$, the following inequalities hold:
\begin{align*}
\frac{a_0}{2}+\sum_{k=1}^n a_k \cos{k\theta}>0 \quad and \quad \sum_{k=1}^n a_k \sin{k\theta}>0,
\quad \mbox{where $0<\theta<\pi$}.
\end{align*}
\end{corollary}

The main purpose of this note is to use Corollary \ref{cor:new-alpha-beta-positive-sum}
to find certain geometric properties of analytic functions,
in particular univalent functions.
Let $\mathcal{A}_0$ be the subclass of the class of analytic functions
$f\in\mathcal{A}$ with normalized condition $f(0)=0,f'(0)=1$
in the unit disc $\mathbb{D}=\{z\in\mathbb{D}, |z|<1\}$..
The subclasses of $\mathcal{A}_0$ consisting of univalent
function is denoted by $\mathcal{S}$. Several subclasses
of univalent functions play a prominent role in the theory
of univalent functions.
For $0\leq \gamma<1$, let $\mathcal{S}^{\ast}(\gamma)$ be
the family of functions $f$ starlike of order $\gamma$ i.e.
if $f\in\mathcal{A}_0$ satisfies the analytic characterization,
\begin{align*}
f\in\mathcal{S}^{\ast}(\gamma)\Longleftrightarrow {\rm{Re}}
\left(\frac{z f'(z)}{f(z)} \right) > \gamma,\hbox{\quad for $ z\in \mathbb{D}.$}
\end{align*}
For $0\leq \gamma<1$, let $C(\gamma)$ be the family of functions
$f$ convex of order $\gamma$ i.e.
if $f\in\mathcal{A}_0$ satisfies the analytic characterization,
\begin{align*}
f\in\mathcal{C}(\gamma)\Longleftrightarrow {\rm{Re}}
\left( 1+ \frac{z f''(z)}{f'(z)} \right) > \gamma,\hbox{\quad for $ z\in \mathbb{D}.$}
\end{align*}
These two classes are related by the Alexander transform,
$f\in\mathcal{C}(\gamma)
\Longleftrightarrow zf'\in\mathcal{S}^{\ast}(\gamma)$. The usual classes
of starlike functions (with respect to origin) and convex functions are denoted respectively by
$\mathcal{S}^{\ast}(0)\equiv\mathcal{S}^{\ast}$ and
$\mathcal{C}(0)\equiv \mathcal{C}$.
An analytic function $f$ is said to be close-to-convex of order $\gamma$,
$(0\leq\gamma<1)$ with respect to a fixed starlike
function $g$ if and only if, it satisfies the analytic characterization,
\begin{align*}
{\mathrm{Re}}\, \, e^{i\eta}\left( \frac{zf'(z)}{g(z)}-\mu\right)>0,\quad
\hbox{$z\in\mathbb{D}$, $\eta\in(-\pi/2,\pi,2)$, $g\in \mathcal{S}^{\ast}$}.
\end{align*}
The family of all close-to-convex
function of order $\mu$ with respect to $g\in\mathcal{S}^{\ast}$ is denoted by $\mathcal{K}_g(\mu)$.
Further, for $0\leq \mu<1$, for each $f\in\mathcal{K}_g(\mu)$ is also univalent in $\mathbb{D}$.
The proper inclusion between these classes is given by
\begin{align*}
\mathcal{C} \subsetneqq \mathcal{S}^{\ast} \subsetneqq \mathcal{K} \subsetneqq \mathcal{S}.
\end{align*}
Another important subclass is the class of typically real functions .
A function $f\in \mathcal{A}_0$ is typically real if $\mathrm {Im}(z)\mathrm{Im}(f(z))\geq 0$
where $z\in\mathbb{D}$.
Its class is denoted by $\mathcal{T}$.
For several interesting geometric properties of these classes, one can refer
to the standard monographs \cite{duren-1983-book,goodman-1983-book,pommerenke-1975-book}
on univalent functions.

\begin{remark}\label{remark:starlike-functions}
The functions
\begin{align*}
z, \quad \frac{z}{1\pm z}, \quad \frac{z}{1\pm z^2}, \quad \frac{z}{(1\pm z)^2}, \quad \frac{z}{(1\pm z+z^2)}
\end{align*}
are the only nine starlike univalent functions having integer coefficients in $\mathbb{D}$.
It will be interesting to find $f$ to be close-to-convex when the corresponding starlike
function $g$ takes one of the above form.
\end{remark}

If we take $\eta=0$ and $g(z)=\frac{z}{(1-z)^2}$ then $\mathrm{Re}((1-z)^2f'(z))>0$ which implies
$zf'(z)$ is typically real function. A function $f\in\mathcal{A}_0$ is said to typically real if
$\mathrm{Im}\,f(z)\mathrm{Im}(z)>0$ whenever $\mathrm{Im}(z)\neq0$, $z\in\mathbb{D}$.
The function $k_{\gamma}(z):=\frac{z}{(1-z)^{2-2\gamma}}$
is the extremal function for the class of
starlike function of order $\gamma$.
Note that $k_0(z)$ is the well-known Koebe function and the function $k_{1/2}(z)=z/(1-z)$
is the extremal function for the class $\mathcal{C}$. A function $f(z)$ is said to be pre-starlike of order
$\gamma$, $0\leq \gamma <1$, if $k_{\gamma}(z)\ast f(z)=\frac{z}{(1-z)^2}\ast f(z)\in\mathcal{S}^{\ast}(\gamma)$
where '$\ast$' is the convolution operator or Hadamard product.
This class was introduced by Ruscheweyh\cite{ruscheweyh-1977-prestarlike} . For more details of this class see
\cite{ruscheweyh-1982-book}.
Here the Hadamard product or convolution is defined as follows:
Let $f(z)=\displaystyle\sum_{k=0}^{\infty}a_kz^k$
and $g(z)=\displaystyle\sum_{k=0}^{\infty}b_kz^k$, $z\in\mathbb{D}$.
Then,
\begin{align*}
(f\ast g)(z)=\sum_{k=0}^{\infty}a_kb_kz^k, \quad z\in\mathbb{D}.
\end{align*}

Among all applications of positivity of trigonometric polynomials, the geometric properties
of the subclasses of analytic functions are considered in this note. In this direction, Ruscheweyh
\cite{ruscheweyh-1987-viet-cond-starlike-func} obtained some coefficient conditions for the
class of starlike functions using the classical result of Vietoris \cite{vietoris-1958}.
So it would be interesting to find the geometric properties of function $f(z)$ in which
Corollary \ref{cor:new-alpha-beta-positive-sum} plays a vital role.

%%%%%%%%%%%%%%%%%%%%%%%%%%%%%%%%%%%%%%%%%%%%%%%%%%%%%%%%%%%%%%%%%
\section{Geometric properties of an analytic function}
\label{sec:2-geom-prop}
%%%%%%%%%%%%%%%%%%%%%%%%%%%%%%%%%%%%%%%%%%%%%%%%%%%%%%%%%%%%%%%%%
In this section, we provide conditions on the Taylor coefficients of an analytic function
$f$ to guarantee the admissibility of $f$ in subclasses of $\mathcal{S}$, using Corollary
\ref{cor:new-alpha-beta-positive-sum}.
The next lemma which is the generalization of \cite [Lemma 2]{ruscheweyh-1987-viet-cond-starlike-func}
is the crucial ingredient in the proof of the following theorem.
\begin{lemma}\rm(\cite{saiful-swami-2011-CAMWA},Theorem 3.1)\label{lemma:saiful-2011-CAMWA-starlike}
Let $0\leq\gamma<1$ and $f\in\mathcal{A}$ be such that $f'(z)$ and $f'(z)-\gamma\frac{f(z)}{z}$ are typically
 real in $\mathbb{D}$. Further if Re$f'(z)>0$ and Re$(f'(z)-\gamma\frac{f(z)}{z})>0$, then $f\in\mathcal{S}^{\ast}(\gamma)$.
\end{lemma}
\begin{theorem}\label{thm:new-positivity-starlike}
Let $\alpha\geq0,\beta\geq0$, $\lambda\geq0,\mu\geq0$ such that $\lambda+\mu\geq1$, let $\{a_k\}_{k=1}^{\infty}$
be any sequence of positive real numbers such that $a_1=1$. Let $\{a_k\}$
satisfy the following conditions:
\begin{enumerate}
\item $(2-\gamma)a_2\leq (1-\gamma)a_1$,
\item $(3-\gamma)a_3 \leq \frac{1}{(2+\alpha)^{\lambda}(2+\beta)^{\mu}}(2-\gamma)a_2$,
\item $(k+2-\gamma)a_{k+2} \leq \left(1+\frac{1}{k+\alpha}\right)^{-\lambda}\left(1+\frac{1}{k+\beta}\right)^{-\mu}
(k+1-\gamma)a_{k+1}, \quad \forall k\geq2$.
\end{enumerate}
Then for $0\leq\gamma<1$, $f_n(z)=z+\displaystyle\sum_{k=1}^n a_kz^k$ and
$f(z)=\displaystyle\lim_{n\to \infty}f_n(z)=z+\displaystyle\sum_{k=2}^{\infty} a_k z^k$ are starlike of order $\gamma$.
\end{theorem}
\begin{proof}
Let $f_n(z)=z+\displaystyle\sum_{k=2}^{n} a_k z^k$, $z\in\mathbb{D}$
be the partial sum of $f$. Then $f_n'(z)=1+\displaystyle\sum_{k=1}^{n-1} (k+1)a_{k+1} z^k$.\\
Define
\begin{align*}
g_n(z):=f_n'(z)-\gamma\frac{f_n(z)}{z}=\frac{b_0}{2}+\sum_{k=1}^{n-1}b_kz^k, \quad z\in\mathbb{D},
\end{align*}
where $b_0=2(1-\gamma)$ and $b_k=(k+1-\gamma)a_{k+1}$ , $\forall k\geq 1$.
Consider,
\begin{align*}
\frac{b_0}{2}- b_1 =(1-\gamma)a_1-(2-\gamma)a_2\geq 0,
\end{align*}
and
\begin{align*}
\frac{b_1}{(2+\alpha)^{\lambda}(2+\beta)^{\mu}}-b_2
=\frac{1}{(2+\alpha)^{\lambda}(2+\beta)^{\mu}}(2-\gamma)a_2-(3-\gamma)a_3\geq0.
\end{align*}
Now for $k\geq 2$,
\begin{align*}
&(k+\alpha)^{\lambda}(k+\beta)^{\mu}b_k-(k+1+\alpha)^{\lambda}(k+1+\beta)^{\mu}b_{k+1}
=(k+1+\alpha)^{\lambda}(k+1+\beta)^{\mu}\\
&\quad\quad\times\left[
\left(1+\frac{1}{k+\alpha}\right)^{-\lambda}\left(1+\frac{1}{k+\beta}\right)^{-\mu}
(k+1-\gamma)a_{k+1}-(k+2-\gamma)a_{k+2}\right]\geq0
\end{align*}
So by the given hypothesis, $\{b_k\}$ satisfy the conditions of Corollary \ref{cor:new-alpha-beta-positive-sum}
which implies $\mathrm{Re}\,g_n(z)>0$ and $\mathrm{Im} \,g_n(z)>0$ if $\mathrm{Im}(z)>0$.
By reflection principle $\mathrm{Im}\, g_n(z)<0$ if $\mathrm{Im} (z)<0$. So $g_n(z)$ is typically
real function. In order to prove the theorem it is remaining to show that $\mathrm{Re} f_n'(z)>0$ and $f_n'(z)$ is typically real.
 In this case $b_k=(k+1)a_{k+1}$ and $b_0=2$.  So such $b_k$ also satisfy given hypothesis because
 $\frac{k+1-\gamma}{k+2-\gamma}<\frac{k+1}{k+2}$ ,
 for all $k\geq0$. So $\mathrm{Re} f_n'(z)>0$ and again using reflection principle we get that
 $f_n'(z)$ is typically real in $\mathbb{D}$.

 Applying Lemma \ref{lemma:saiful-2011-CAMWA-starlike}, we get that $f_n(z)\in\mathcal{S}^{\ast}(\gamma)$.
 Since $\displaystyle\lim_{n\to\infty}f_n(z)=f(z)$ and
 the family of starlike functions is normal \cite[p.217]{nehari}, we get
 $f(z)=\displaystyle\lim_{n\rightarrow\infty}f_n(z)$ is also starlike of order $\gamma$.
\end{proof}
\begin{remark}
If $\gamma=0$ in Theorem \ref{thm:new-positivity-starlike}, then we get $\mathrm{Re}(f_n'(z))>0$
which implies $f_n(z)$ is close-to-convex with respect to $z$ and $f_n'(z)$ is typically real
also and with
\begin{align*}
\mathrm{Re}(1-z)f_n'(z)=\mathrm{Re} (1-z)\mathrm{Re} f_n'(z)+\mathrm{Im}(z)\mathrm{Im} f_n'(z)>0
\end{align*}
this yields $f_n(z)$ is close-to-convex with respect to starlike function $z/(1-z)$.
\end{remark}

\begin{example}\label{exp:starlike}
Consider the sequence $\{a_k\}$ as $a_1=1$, $a_2=\frac{1}{2}$ and $a_k=\frac{1}{k^2}$
for $k\geq3$, then by Theorem \ref{thm:new-positivity-starlike} the function
\begin{align*}
f(z)=z+\frac{z^2}{2}+\sum_{k=3}^{\infty}\frac{z^k}{k^2},\quad z\in\mathbb{D},
\end{align*}
is starlike univalent. But \cite[Theorem 2.1]{sangal-swaminathan-starlike-HGF} fails to include this function.
Hence Theorem \ref{thm:new-positivity-starlike} is better than \cite[Theorem 2.1]{sangal-swaminathan-starlike-HGF}
in the sense that it is likely to include more cases.
\end{example}

By proving that $zf_n'(z)$ is typically real function in the similar fashion,
we obtain the next result.

\begin{theorem}\label{thm:new-positivity-close-to-convex}
Let $\alpha\geq0,\beta\geq0$ and $\lambda\geq0,\mu\geq0$ such that $\lambda+\mu\geq1$,
let $\{a_k\}_{k=1}^{\infty}$
be any sequence of positive real numbers such that $a_1=1$. If $\{a_k\}$
satisfy the following conditions:
\begin{align*}
(k+1+\alpha)^{\lambda}(k+1+\beta)^{\mu}(k+1)a_{k+1} &\leq (k+\alpha)^{\lambda}(k+\beta)^{\mu}
ka_k\\
&\leq \cdots \leq (2+\alpha)^{\lambda}(2+\beta)^{\mu}2a_2 \leq 1 ,\quad \hbox{for $k\geq 2$}.
\end{align*}
Then  $f_n(z)=z+\displaystyle\sum_{k=2}^n a_kz^k$ and $f(z)=z+\displaystyle\sum_{k=2}^{\infty}a_kz^k$
 are close-to-convex with respect to starlike function $z/(1-z^2)$.
\end{theorem}

Note that Theorem \ref{thm:new-positivity-close-to-convex}
provides close-to-convexity of $f$ with respect to the
function $z/(1-z^2)$. Results for the close-to-convexity of $f$ with respect to other
four starlike functions given in Remark \ref{remark:starlike-functions}
are of considerable interest, and the authors have considered some of these cases separately elsewhere.
The next result provides the coefficient conditions for $f$ to be in the class of prestarlike
functions of order $\gamma$, $0\leq \gamma<1$.

\begin{theorem}\label{thm:prestarlike}
Let $\alpha\geq0,\beta\geq0$, $\lambda\geq0,\mu\geq0$ such that
$\lambda+\mu\geq1$, let $\{a_k\}_{k=1}^{\infty}$
be any sequence of positive real numbers such that $a_1=1$. Let $\{a_k\}$ satisfy
the following conditions:
\begin{enumerate}
%\item $2(2-\gamma)a_2 \leq a_1$
\item $(2+\alpha)^{\lambda}(2+\beta)^{\mu}(3-\gamma)(3-2\gamma)a_3 \leq 2(2-\gamma)a_2\leq a_1$,
\item $(k+1+\alpha)^{\lambda}(k+1+\beta)^{\mu}(k+2-\gamma)(k+2-2\gamma)a_{k+2}
\leq (k+\alpha)^{\lambda}(k+\beta)^{\mu} (k+1-\gamma)(k+1)a_{k+1}, \quad \forall k\geq2$.
\end{enumerate}
Then for $0\leq\gamma<1$, $f_n(z)=z+\displaystyle\sum_{k=2}^n a_kz^k$ is prestarlike of order $\gamma$.
Moreover,
$f(z)=z+\displaystyle\sum_{k=2}^{\infty} a_k z^k$ is prestarlike of order $\gamma$.
\end{theorem}
\begin{proof}
Let $g_n(z):=f_n(z)\ast\frac{z}{(1-z)^{2-2\gamma}}$, $z\in\mathbb{D}$, $0\leq \gamma<1$.
To prove required theorem it is sufficient
to prove that $g_n(z)\in\mathcal{S}^{\ast}(\gamma)$.
\begin{align*}
g_n(z)=f_n(z) \ast \frac{z}{(1-z)^{2-2\gamma}} = z+\sum_{k=2}^n
 \frac{(2-2\gamma)_{k-1}}{(k-1)!}a_kz^k, \quad z\in\mathbb{D}.
\end{align*}
We prove that $g_n(z)$ satisfy the conditions of
Lemma \ref{lemma:saiful-2011-CAMWA-starlike}. For this, define
\begin{align*}
 h_n(z):=g_n'(z)-\gamma\frac{g_n(z)}{z}=\frac{b_0}{2}+\sum_{k=1}^{n-1}b_kz^k, \quad z\in\mathbb{D},
\end{align*}
where $b_0=2(1-\gamma)$ and $b_k=(k+1-\gamma)\frac{(2-2\gamma)_k}{k!}a_{k+1}$ for $k\geq 1$.
Using simple calculations, along with the hypothesis, $\{b_k\}$ satisfy the conditions of
Corollary \ref{cor:new-alpha-beta-positive-sum}.
Continuing the same argument as earlier, we get the desired result.
\end{proof}

\begin{remark}
Note that $\mathcal{R}(1/2)=\mathcal{S}^{\ast}(1/2)$. It can be easily verified that all the conditions of
Theorem \ref{thm:prestarlike} for $\mathcal{R}(1/2)$ coincide with the conditions of Theorem \ref{thm:new-positivity-starlike}
for $\mathcal{S}^{\ast}(1/2)$.
\end{remark}

For $\gamma=0$, $\mathcal{R}^{\ast}(0)\equiv \mathcal{C}$ and the following result is immediate.
\begin{corollary}
For $\alpha\geq0,\beta\geq0$, $\lambda\geq0,\mu\geq0$ such that
$\lambda+\mu\geq1$, let $\{a_k\}_{k=1}^{\infty}$
be any sequence of positive real numbers such that $a_1=1$. Let $\{a_k\}$
satisfy the following condition
\begin{align*}
(k+1+\alpha)^{\lambda}(k+1+\beta)^{\mu}(k+2)^2a_{k+2}&\leq (k+\alpha)^{\lambda}(k+\beta)^{\mu}
(k+1)^2a_{k+1}\\
&\leq\cdots\leq(2+\alpha)^{\lambda}(2+\beta)^{\mu}9 a_3 \leq 4 a_2\leq a_1.
\end{align*}
Then $f_n(z)=z+\displaystyle\sum_{k=2}^n a_kz^k$ is convex function. In particular
$f(z)=z+\displaystyle\sum_{k=2}^{\infty} a_k z^k$ is convex univalent.
\end{corollary}

\begin{example}
Let $f(z)=z+\frac{z^2}{4}+\displaystyle\sum_{k=3}^{\infty}\frac{z^k}{(k-1+\alpha)^{\lambda}(k-1+\beta)^{\mu}k^2}$
is convex univalent.\\
In particular if $\alpha=\beta=1$ and $\lambda=\mu=1/2$, we get that
$z+\frac{z^2}{4}+\displaystyle\sum_{k=3}^{\infty}\frac{z^k}{k^3}$
is convex.
\end{example}

%%%%%%%%%%%%%%%%%%%%%%%%%%%%%%%%%%%%%%%%%%%%%%%%%
\section{Application to Ces\`aro mean of type $(b-1,c)$}
%%%%%%%%%%%%%%%%%%%%%%%%%%%%%%%%%%%%%%%%%%%%%%%%%

The nth Ces\`aro mean of type $(b-1,c)$ of $f(z)\in\mathcal{A}_0$ is given by,
\begin{align}\label{eqn:def-sigma-bc}
s_n^{(b-1,c)}(z,f):=z+\sum_{k=2}^n \frac{B_{n-k}}{B_{n-1}} a_k z^k=s_n^{(b-1,c)}(z)\ast f(z) , \quad \hbox{$n\in\mathbb{N}$},
\end{align}
where $b$ and $c$ are real numbers such that $b+1>c>0$ and $B_0=1$
and $B_k=\frac{(1+b-c)}{b}\frac{(b)_k}{(c)_k}$ for $k\geq 1$.
Here by $(\alpha)_k$, $k\geq \mathbb{N}$, which is the well-known Pochhammer symbol, we mean the following:
\begin{align*}
(\alpha)_k=\alpha(\alpha+1)_{k-1}\quad  \hbox{with} \quad  (\alpha)_0=1.
\end{align*}
For $b=1+\delta$ and $c=1$, it follows that,
\begin{align*}
s_n^{(\delta,1)}(f,z)=s_n^{\delta}(f,z)=z+\sum_{k=2}^{n}\frac{(1+\delta)_{n-k}}{(n-k)!}\frac{(n-1)!}{(1+\delta)_{n-1}}a_kz^k,
\end{align*}
which is the Ces\`{a}ro mean of order $\delta$ for $\delta>-1$.
Since \eqref{eqn:def-sigma-bc} is one type of generalization of the well-known Ces\`aro mean
\cite{ruscheweyh-1992-geom-cesaro-result-in-math}
we call these Ces\`aro mean of
type $(b-1;c)$ as  generalized Ces\`aro operators.
The coefficients given in \eqref{eqn:def-sigma-bc} were considered in
\cite{sangal-swaminathan-positivity-alpha-beta} while finding positivity of
trigonometric polynomials. Using \eqref{eqn:def-sigma-bc} generalized Ces\`aro
averaging operators were studied in \cite{ponnusamy-naik-gen-means-2004}
which are generalization of the Ces\`aro operator given by Stempak \cite{stempak-cesaro-1994-Edinburg}.
The geometric properties of $s_n^{\delta}(z)$ are well-known. For details, see
\cite{ali-saiful-zero-free-2013,bustoz-1975-AMS,ruscheweyh-1992-geom-cesaro-result-in-math}.
Lewis \cite{lewis-1979-convolution-jacobi-SIAM}
proved that $s_n^{\delta}(z)$ is close-to-convex and hence univalent for
$\delta\geq1$. Ruscheweyh \cite{ruscheweyh-1992-geom-cesaro-result-in-math}
proved that it is prestarlike of order $(3-\delta)/2$.
Hence it would be interesting to see if the geometric properties of $s_n^{\delta}(z)$
can be extended to $s_n^{(b-1,c)}(f,z)$.
Such investigations are possible by various known methods in Geometric
function theory. In particular, the positivity techniques used in
Koumandos \cite{koumandos-2007-ext-viet-ramanujan} or Saiful and Swaminathan \cite{saiful-swami-2011-CAMWA}
can be applied to $s_n^{(b-1,c)}(z)$ as well. However, in view of Example \ref{exp:starlike},
we are interested in using the results available in Section \ref{sec:2-geom-prop}
to obtain the geometric properties of $s_n^{(b-1,c)}(z)$.

\begin{theorem}\label{thm:close-to-convex-generalized-mean}
Let $\{a_k\}$ be any sequence of positive real numbers such that $a_1=1$ and
$(b+n-2)a_1\geq2(c+n-2)a_2$. Let $b\geq c>0$, $0\leq \alpha\leq
\frac{6}{\lambda+4},0\leq \beta\leq \frac{6}{\mu+4}$ and $\lambda,\mu\geq0$ such that
$\lambda+\mu\geq1$ and $1\leq\lambda+\mu<2 $ and satisfies the following conditions:
\begin{itemize}
\item[(i)] $(2-\alpha\lambda)(2-\beta\mu)(b+n-3)a_2 \geq 2^{\lambda+\mu+1}(c+n-3)3a_3$
\item[(ii)] $(k-1+\alpha-\lambda)(k-1+\beta-\mu)(b+n-k-1)ka_k \geq (k-1+\alpha)
            (k-1+\beta)(c+n-k-1)(k+1)a_{k+1}$ for $3\leq k\leq n-3$
\item[(iii)] $(n-2+\alpha-\lambda)(n-2+\beta-\mu)(1+b-c)(n-1)a_{n-1}\geq
            (n-2+\alpha)(n-2+\beta)c n a_n$
\end{itemize}
Then $s_n^{(b-1,c)}(f,z)$ is close-to-convex with respect to $z$ and $\frac{z}{1-z}$
where $f(z)=z+\displaystyle\sum_{k=2}^{\infty}a_kz^k$.
Further for the same condition $s_n^{(b-1,c)}(f,z)$ is starlike univalent.
\begin{proof}
Let $s_n^{(b-1,c)}(f,z)=z+\displaystyle\sum_{k=2}^n \frac{B_{n-k}}{B_{n-1}}a_kz^k$. Then,
\begin{align*}
s_n^{(b-1,c)}(f,z)'= 1+\sum_{k=1}^{n-1} \frac{B_{n-k-1}}{B_{n-1}}(k+1)a_{k+1}z^k.
\end{align*}
For $0\leq r<1$ and $0\leq \theta\leq 2\pi$,
\begin{align*}
\mathrm {Re}\, s_n^{(b-1,c)}(f,z)' &=\frac{b_0}{2}+\sum_{k=1}^{n-1} b_kr^k\cos{k\theta},\\
\mathrm {Im}\, s_n^{(b-1,c)}(f,z)' &=\sum_{k=1}^{n-1}b_k r^k\sin{k\theta},
\end{align*}
where $b_0=2$ and $b_k=\frac{B_{n-k-1}}{B_{n-1}}(k+1)a_{k+1}$ for $k\geq1 $.
Hence $b_k$ and $b_{k+1}$ can be related as:
\begin{align*}
b_{k+1}=\frac{(c+n-k-2)(k+2)a_{k+2}}{(b+n-k-2)(k+1)a_{k+1}}b_k ,\hbox{\quad for $1\leq k\leq n-3$},
\end{align*}
and
\begin{align*}
b_{n-1}=\frac{c}{(1+b-c)}\frac{na_n}{(n-1)a_{n-1}}b_{n-2}.
\end{align*}
For the sequence $\{b_k\}$, our  aim is to prove that,
$\frac{b_0}{2}+\displaystyle\sum_{k=1}^{n-1} b_k r^k \cos{k\theta}>0$ and
$\displaystyle\sum_{k=1}^{n-1}b_k r^k\sin{k\theta}>0$.
Note that,
\begin{align*}
\frac{b_0}{2}-b_1=\frac{1}{(b+n-2)}[(b+n-2)a_1-2(c+n-2)a_2]\geq0.
\end{align*}
For a given $\alpha$ and $\beta$, we can easily get,
\begin{align*}
&\frac{1}{(2+\alpha)^{\lambda}(2+\beta)^{\mu}}= \frac{1}{2^{\lambda+\mu}}\left(1
+\frac{\alpha}{2}\right)^{-\lambda}\left(1+\frac{\beta}{2}\right)^{-\mu}\\
&=\frac{1}{2^{\lambda+\mu}}\left(1-\frac{\alpha\lambda}{2}+\frac{\lambda(\lambda+1)}{2!}\frac{\alpha^2}{2^2}-
\frac{\lambda(\lambda+1)(\lambda+2)}{3!}\frac{\alpha^3}{2^3}+\frac{\lambda(\lambda+1)(\lambda+2)(\lambda+3)}{4!}
\frac{\alpha^4}{2^4}-\cdots \right)\\
& \times \left( 1-\frac{\beta\mu}{2}+\frac{\mu(\mu+1)}{2!}\frac{\beta^2}{2^2}-
\frac{\mu(\mu+1)(\mu+2)}{3!}\frac{\beta^3}{2^3}+\frac{\mu(\mu+1)(\mu+2)(\mu+3)}{4!}
\frac{\beta^4}{2^4}-\cdots\right)\\
&=\frac{1}{2^{\lambda+\mu+2}}\bigg[(2-\alpha\lambda)+\lambda(\lambda+1)\frac{\alpha^2}{2^2}\left(1-
\frac{(\lambda+2)}{6}\alpha\right)\\
&\quad +\frac{\lambda(\lambda+1)(\lambda+2)(\lambda+3)}{3\cdot2^6}
\alpha^4\left(1-\frac{(\lambda+4)}{10}\alpha\right)+\cdots \bigg]\times \bigg[(2-\beta\mu)\\
&+\mu(\mu+1)\frac{\beta^2}{2^2}\left(1-\frac{(\mu+2)}{6}\beta\right)+\frac{\mu(\mu+1)(\mu+2)(\mu+3)}{3\cdot2^6}
\beta^4\left(1-\frac{\mu+4}{10}\beta\right)+\cdots\bigg]\\
&\geq \frac{(2-\alpha\lambda)(2-\beta\mu)}{2^{\lambda+\mu+2}} \quad
\hbox{if $0\leq \alpha \leq \frac{6}{\lambda+2}$ and $0\leq \beta\leq \frac{6}{\mu+2}$}.
\end{align*}
Hence we see that,
\begin{align*}
&\frac{b_1}{(2+\alpha)^{\lambda}(2+\beta)^{\mu}}-b_2
\geq \frac{(2-\alpha\lambda)(2-\beta\mu)b_1}{2^{\lambda+\mu+2}}-\frac{(c+n-3)3a_3}{(b+n-3)2a_2}b_1\geq0.
\end{align*}
For the other condition
$(k+\alpha)^{\lambda}(k+\beta)^{\mu}b_k \geq (k+1+\alpha)^{\lambda}(k+1+\beta)^{\mu}b_{k+1}$
to be satisfied, first we find,
\begin{align*}
&\left[ 1+\frac{1}{k+\alpha}\right]^{-\lambda}\left[ 1+\frac{1}{k+\beta}\right]^{-\mu}\\
= &\bigg[ 1-\frac{\lambda}{k+\alpha}+\frac{\lambda(\lambda+1)}{2!(k+\alpha)^2}\left(1-\frac{(2+\lambda)}{3.(k+\alpha)}\right)+
   \frac{\lambda(\lambda+1)(\lambda+2)(\lambda+3)}{4!(k+\alpha)^4}\left(1-\frac{\lambda+4}{5(k+\alpha)} \right)\\
   &+\cdots\bigg]
\times  \bigg[1-\frac{\mu}{k+\beta}+\frac{\mu(\mu+1)}{2!(k+\beta)^2}\left(1-\frac{(2+\mu)}{3.(k+\beta)}\right)+\\
&\quad \quad    \frac{\mu(\mu+1)(\mu+2)(\mu+3)}{4!(k+\beta)^4}\left(1-\frac{\mu+4}{5(k+\beta)} \right)+\cdots \bigg]
\end{align*}
\begin{align*}
\geq & \left(1-\frac{\lambda}{k+\alpha}\right)\cdot \left(1-\frac{\mu}{k+\beta}\right),\quad
\hbox{if $\frac{2+\lambda}{3(k+\alpha)}\leq1$ and $\frac{2+\mu}{3(k+\beta)}\leq1$ for $k\geq2$}.
\end{align*}
Clearly,
\begin{align*}
&\left(1+\dfrac{1}{k+\alpha}\right)^{-\lambda} \left(1+\dfrac{1}{k+\beta}\right)^{-\mu} b_k - b_{k+1}\\
&\geq \left(1-\dfrac{\lambda}{k+\alpha}\right) \left(1-\dfrac{\mu}{k+\beta}\right) b_k
- \frac{(c+n-k-2)(k+2)a_{k+2}}{(b+n-k-2)(k+1)a_{k+1}}b_k ,\hbox{\quad for $2 \leq k\leq n-3$}\\
&=b_{k-1}\left[\left(1-\dfrac{\lambda}{k-1+\alpha}\right) \left(1-\dfrac{\mu}{k-1+\beta}\right)
-\frac{(c+n-k-1)(k+1)a_{k+1}}{(b+n-k-1)k a_k}\right] \\
&\geq0,\hbox{ for $3 \leq k\leq n-2$}.
\end{align*}
For $k=n-2$, consider
\begin{align*}
&\dfrac{1}{\left(1+\dfrac{1}{n-2+\alpha}\right)^{\lambda}\left(1+\dfrac{1}{n-2+\beta}\right)^{\mu}}b_{n-2}-b_{n-1}\\
&\geq \left(1-\dfrac{\lambda}{n-2+\alpha}\right)\left(1-\dfrac{\mu}{n-2+\beta}\right)\frac{(1+b-c)}{c}(n-1)a_{n-1}-na_n
\geq0.
\end{align*}
We proved that $\frac{b_0}{2}+\displaystyle\sum_{k=1}^{n-1} b_k\cos{k\theta}>0$
and $\displaystyle\sum_{k=1}^{n-1} b_k \sin{k\theta}>0$ for $0<\theta<\pi$.
 By the minimum principle for harmonic functions,
 $\frac{b_0}{2}+\displaystyle\sum_{k=1}^{n-1} b_kr^k\cos{k\theta}>0$,
 $0\leq r<1$ and $0<\theta<\pi$ and $\displaystyle\sum_{k=1}^{n-1} b_k r^k \sin{k\theta}>0$
 for $0<\theta<\pi$ and $0\leq r<1$. Using reflection principle,
 $\displaystyle\sum_{k=1}^{n-1} b_k r^k \sin{k\theta}<0$ for $\pi<\theta<2\pi$  and $0\leq r<1$.
 Note that $s_n^{(b-1,c)}(f,z)$ is close to convex with respect to
$z$ if $\mathrm{Re}\,s_n^{(b-1,c)}(f,z)>0$
and $s_n^{(b-1,c)}(f,z)$ is close to convex with respect to
$\frac{z}{1-z}$ if $\mathrm{Re}[(1-z)s_n^{(b-1,c)}(f,z)']>0$. Now
 \begin{align*}
\mathrm{Re} [(1-z)s_n^{(b-1,c)}(f,z)'] &= \mathrm{Re}(1-z)\mathrm{Re}(s_n^{(b-1,c)}(f,z)')- \mathrm{Im}(1-z)\mathrm{Im}(s_n^{(b-1,c)}(f,z)')\\
 &=\mathrm{Re}(1-z)\mathrm{Re}(s_n^{(b-1,c)}(f,z)')+\mathrm{Im}(z)\mathrm{Im}(s_n^{(b-1,c)}(f,z)')>0.\qedhere
\end{align*}
\end{proof}
\end{theorem}

For $b=1+\delta, c=1$, Theorem \ref{thm:close-to-convex-generalized-mean} leads to the following example.
\begin{example}
Let $\lambda\geq0$, $\mu\geq0$ such that $1\leq\lambda+\mu<2, 0\leq\alpha\leq\frac{6}{\lambda+4}$
and $0\leq \beta\leq \frac{6}{\beta+4}$ then
\begin{align*}
\delta\geq \max_{n\geq1} \left\{0,(n-2)\left(\frac{2^{\lambda+\mu+2}}{(2-\alpha\lambda)(2-\beta\mu)}-1 \right),
(n-3)\left( \frac{2(\lambda+\mu)+\alpha\mu+\beta\lambda+\lambda\mu}{(2+\alpha-\lambda)(2+\beta-\mu)} \right) \right\}
\end{align*}
Then $s_n^{\delta}(-\log{(1-z)},z)$ is close-to-convex with respect to $z$ and $z/(1-z)$.
Further for the same condition it is also starlike univalent.
\end{example}

\begin{remark}
If we take $\alpha=\beta=1$, and $\lambda=\mu=\frac{1}{2}$ then for $1\leq n\leq 3$ ,$s_n^{\delta}(-\log{(1-z)},z)$ is close-to-convex
with respect to $z$ and $z/(1-z)$ for $\delta\geq \delta'$ where $0<\delta'<3$. This conclusion cannot be obtained from \cite[Corollary 4.2]{saiful-swami-2011-CAMWA}.
\end{remark}

\begin{theorem}\label{thm:prestarlike-sa}
Let $\{a_k\}$ be a sequence of positive real numbers with $a_1=1$ and satisfy
the hypothesis of Theorem \ref{thm:close-to-convex-generalized-mean}. Then
$s_n^{(b-1,c)}(f,z)\in \mathcal{R}(\gamma)$, $\gamma\geq0$ where
\begin{align*}
\gamma \leq 1-\frac{(c+n-2)}{(b+n-2)}2a_2,
\end{align*}
$\mathcal{R}(\gamma)=\{f\in\mathcal{A}:\mathrm{Re} f'(z)>\gamma\}$
and $f(z)=z+\displaystyle\sum_{k=2}^{\infty}a_kz^k$, $z\in\mathbb{D}$.
\begin{proof}
Let $s_n^{(b-1,c)}(f,z)=z+\displaystyle\sum_{k=2}^n \frac{B_{n-k}}{B_{n-1}}a_kz^k$ where $B_0=1$ and $B_k=\frac{(b)_k}{(c)_k}\frac{1+b-c}{b},k\geq 1$.
\begin{align*}
s_n^{(b-1,c)}(f,z)'=1+\sum_{k=1}^{n-1} \frac{B_{n-k-1}}{B_{n-1}}(k+1)a_{k+1}z^k.
\end{align*}
We consider %$\mathrm {Re}\frac{\sigma_n^{(b,c)}(z,f)'-\gamma}{1-\gamma}$.
\begin{align*}
\frac{s_n^{(b-1,c)}(f,z)'-\gamma}{1-\gamma}=\frac{b_0}{2}+\sum_{k=1}^{n-1}b_kz^k,
\end{align*}
where $b_0=2$ and $b_k=\frac{B_{n-k-1}}{B_{n-1}}.\frac{(k+1)a_{k+1}}{(1-\gamma)}$ for $k\geq 1$.
Then $b_k$ and $b_{k+1}$ are related by
\begin{align*}
b_{k+1}=\frac{(c+n-k-2)(k+2)a_{k+2}}{(b+n-k-2)(k+1)a_{k+1}}b_k \hbox{\quad for $1\leq k\leq n-3$},
\end{align*}
and for $k=n-2$,
\begin{align*}
b_{n-1} &=\frac{(1+b-c)}{c}\frac{(n-1)a_{n-1}}{na_n}b_{n-2}.
\end{align*}
Using hypothesis we can easily get,
\begin{align*}
\frac{b_0}{2}-b_1=1-\left(\frac{c+n-2}{b+n-2}\right)\frac{2a_2}{1-\gamma}\geq0.
\end{align*}
The relation between the coefficients $b_k$ and $b_{k+1}$ is same as in the
Theorem \ref{thm:close-to-convex-generalized-mean}.
So such $b_k $ also satisfy the conditions of Theorem \ref{thm:close-to-convex-generalized-mean}
and from Corollary \ref{cor:new-alpha-beta-positive-sum}
we have the required result that
\begin{align*}
\frac{b_0}{2}+\sum_{k=1}^{n-1}b_k\cos{k\theta}>0 \quad \hbox{for $0<\theta<\pi$}.
\end{align*}
From the minimum principle for harmonic functions for $0\leq r<1$ and $0<\theta<2\pi$ we have
\begin{align*}
\mathrm {Re}\left( \frac{s_n^{(b-1,c)}(f,z)'-\gamma}{1-\gamma}\right)=\frac{b_0}{2}+\sum_{k=1}^{n-1}b_kr^k\cos{k\theta}>0.
\end{align*}
So, $s_n^{(b-1,c)}(f,z)\in\mathcal{R}(\gamma)$.
\end{proof}
\end{theorem}

It can be clearly seen that for $\gamma=0$, Theorem \ref{thm:prestarlike-sa} coincides with
Theorem \ref{thm:close-to-convex-generalized-mean} for the case $g(z)=z$.

\begin{theorem}
Let $\{ a_k\}_{k=1}^{\infty}$ be a sequence of positive real numbers such that $a_1=1.$ If for
$\lambda\geq0,\mu\geq0$ such that $1\leq \lambda+\mu <2$ and
 $0\leq \alpha\leq \frac{6}{\lambda+4},0\leq\beta\leq\frac{6}{\mu+4}$, $a_k$ satisfy the following conditions:
 \begin{enumerate}
 \item $(3-2\lambda-2\mu)(b+n-2)a_1 \geq (5-2\lambda-2\mu)(c+n-2)a_2$
 \item $(2-\alpha\lambda)(2-\beta\mu)(5-2\lambda-2\mu)(b+n-3)a_2\geq 2^{\lambda+\mu+2}(7-2\lambda-2\mu)(c+n-3)a_3$
 \item $(2k+1-2\lambda-2\mu)(k-1+\alpha-\lambda)(k-1+\beta-\mu)(b+n-k-1)a_{k}\geq (2k+3-2\lambda-2\mu)(k-1+\alpha)
        (k-1+\beta)(c+n-k-1)a_{k+1} \quad \hbox{for $3 \leq k\leq n-2$}$
 \item $(n-2+\alpha-\lambda)(n-2+\beta-\mu)(2n+1-2\lambda-2\mu)(1+b-c)
 a_{n-1}\geq (n-2+\alpha)(n-2+\beta)(2n+3-2\lambda-2\mu)c a_n$,
 \end{enumerate}
 then, $s_n^{(b-1,c)}(f,z)\in \mathcal{S}^{\ast}(\lambda+\mu-1/2)$, where $f(z)=z+\displaystyle\sum_{k=2}^{\infty}a_kz^k$,
 $z\in\mathbb{D}$.
 \end{theorem}
\begin{proof}
$s_n^{(b-1,c)}(f,z)=z+\displaystyle\sum_{k=2}^n \frac{B_{n-k}}{B_{n-1}}a_k z^k
=b_1z+\displaystyle\sum_{k=2}^n b_kz^k$,
where $b_1=1$ and $b_k=\frac{B_{n-k}}{B_{n-1}}a_k$ for $k\geq 2$. Then,
\begin{align*}
b_{k+1}=\left(\frac{c+n-k-1}{b+n-k-1}\right)\frac{a_{k+1}}{a_k}b_k, \hbox{\quad for $2\leq k\leq n-2$},
\end{align*}
and for $k=n-1$, $b_n=\frac{c}{(1+b-c)}\frac{a_n}{a_{n-1}}b_{n-1}$.
It is enough to prove that $\{b_k\}$ satisfy the conditions of Theorem \ref{thm:new-positivity-starlike}.
For the sake of convenience we substitute $\gamma=\lambda+\mu-1/2$.
By a simple calculation we can get that $(1-\gamma)b_1-(2-\gamma)b_2\geq0$.
Now
\begin{align*}
&\frac{1}{(2+\alpha)^{\lambda}(2+\beta)^{\mu}}(2-\gamma)b_2-(3-\gamma)b_3\\
&\geq \frac{(2-\alpha\lambda)(2-\beta\mu)}{2^{\lambda+\mu+2}}(5-2\lambda-2\mu)b_2-
 (7-2\lambda-2\mu)\frac{(c+n-3)}{(b+n-3)}\frac{a_3}{a_2}b_2\geq0.
\end{align*}
Now for $2 \leq k\leq n-3$,
\begin{align*}
&\left(1+\frac{1}{k+\alpha}\right)^{-\lambda}\left(1+\frac{1}{k+\beta}\right)^{-\mu}
(k+1-\gamma)b_{k+1}-(k+2-\gamma)b_{k+2}\\
&\geq \left(1-\frac{\lambda}{k+\alpha}\right)\left(1-\frac{\mu}{k+\beta}\right)(2k+3-2\lambda-2\mu)b_{k+1}
-(2k+5-2\lambda-2\mu)\\
&\quad\quad\left(\frac{c+n-k-2}{b+n-k-2}\right)\frac{a_{k+2}}{a_{k+1}}b_{k+1}\geq0.
\end{align*}
and for $k=n-2$, using the hypothesis we obtain,
\begin{align*}
\left(1+\frac{1}{n-2+\alpha}\right)^{-\lambda}\left(1+\frac{1}{n-2+\beta}\right)^{-\mu}(2n+1-2\lambda-2\mu)b_{n-1}\\
-(2n+3-2\lambda-2\mu)b_n\geq0.
\end{align*}
From Theorem \ref{thm:new-positivity-starlike} the desired result follows.
\end{proof}

\begin{theorem}\label{thm:sigma-b-c-prestarlike}
Let $b\geq c>0$, $0\leq \alpha\leq
\frac{6}{\lambda+4},0\leq \beta\leq \frac{6}{\mu+4}$ and $\lambda,\mu\geq0$ such that
$1\leq\lambda+\mu<2 $ and satisfies the following conditions:
\begin{enumerate}
\item $2(2-\gamma)(c+n-2)\leq (b+n-2)$,
\item $(2+\alpha)^{\lambda}(2+\beta)^{\mu}(3-\gamma)(3-2\gamma)(c+n-3)
\leq 2(2-\gamma)(b+n-3)$,
\item $(k+1+\alpha)^{\lambda}(k+1+\beta)^{\mu}(k+2-\gamma)(k+2-2\gamma)(c+n-k-2)
\leq (k+\alpha)^{\lambda}(k+\beta)^{\mu}(k+1-\gamma)(k+1)(b+n-k-2)$ for $2\leq k\leq n-3$,
\item $(n-1+\alpha)^{\lambda}(n-1+\beta)^{\mu}(n-\gamma)(n-2\gamma)c\leq
(n-2+\alpha)^{\lambda}(n-2+\beta)^{\mu}(n-1-\gamma)(n-1)(1+b-c)$.
\end{enumerate}
Then $s_n^{(b-1,c)}(z)$ is prestarlike of order $\gamma$, where $0\leq \gamma<1$.
\end{theorem}
\begin{proof}
It is given that $s_n^{(b-1,c)}(z)=z+\displaystyle\sum_{k=2}^n \frac{B_{n-k}}{B_{n-1}}z^k
=z+\displaystyle\sum_{k=2}^n a_kz^k$, $z\in\mathbb{D}$. \\
Then using $a_k=\frac{B_{n-k}}{B_{n-1}}$ for $k\geq1$ in Theorem \ref{thm:new-positivity-starlike} and
following the same procedure the result can be proved.
\end{proof}

If $\gamma=0$ then $s_n^{(b-1,c)}(z)\in\mathcal{R}^{\ast}(0)=\mathcal{C}$.
Further if we substitute $b=1+\delta$ and $c=1$ in Theorem \ref{thm:sigma-b-c-prestarlike},
we have the following example.
\begin{example}\label{exp:prestarlike-cesaro}
If $\alpha,\beta,\lambda$ and $\mu$ satisfy the conditions of Theorem \ref{thm:sigma-b-c-prestarlike}
and if
\begin{align*}
\delta\geq\max\biggl\{ (n-1)(3-2\gamma), (n-2)\left(
\frac{(2+\alpha)^{\lambda}(2+\beta)^{\mu}(3-\gamma)(3-2\gamma)}{2(2-\gamma)}-1\right),\\
(n-3)\left( \frac{(3+\alpha)^{\lambda}(3+\beta)^{\mu}(4-\gamma)(4-2\gamma)}
{(2+\alpha)^{\lambda}(2+\beta)^{\mu}(3-\gamma)3}-1\right)\biggr\}.
\end{align*}
Then $s_n^{\delta}(z)$ is prestarlike of order $\gamma$, where $\gamma\in[0,1)$.
\end{example}
It can be noted that if we take $\alpha=\beta=0$ and $\lambda+\mu=1$ in Example \ref{exp:prestarlike-cesaro},
then for $\delta\geq (n-1)(3-2\gamma)$, $s_n^{\delta}(z)\in \mathcal{R}^{\ast}(\gamma)$.
Similar type of result had been found in \cite[Theorem 1]{ruscheweyh-1992-geom-cesaro-result-in-math}.
From \cite[Theorem 2.1]{ali-saiful-zero-free-2013}, we deduce the following corollary.
\begin{corollary}
If $\alpha,\beta,\lambda,\mu$ and $\gamma$ satisfy the hypothesis of Theorem $\mathrm{\ref{thm:sigma-b-c-prestarlike}}$
then for $b\geq c$, $s_n^{(b-1,c)}\in \mathcal{R}^{\ast}(\gamma)$. Then for any $zg\in\mathcal{K}(\gamma)$
$\Rightarrow g\ast(s_n^{(b-1,c)}(z))'$ is zero free in $\mathbb{D}$.
\end{corollary}
The following Lemma, which is the extension of the well-konwn Polya-Schoenberg Theorem is ingredient to our next result.
\begin{lemma}\cite[p. 499]{ruscheweyh-1977-prestarlike}
If $f\in \mathcal{K}(\gamma), g\in\mathcal{R}^{\ast}(\gamma),0\leq \gamma<1$
then $f\ast g\in \mathcal{K}(\gamma)$.
\end{lemma}
Clearly, $s_n^{(b-1,c)}(g,z)\in \mathcal{K}(\gamma)$ if $g\in\mathcal{K}(\gamma)$.

\begin{theorem}
Let $\{a_k\}$ be a sequence of positive real numbers such that $a_1=1$.
Then  for $0\leq\alpha\leq\frac{6}{\lambda+4},0\leq \beta\leq \frac{6}{\mu+4}$ and
$\lambda\geq0,\mu\geq0$ such that $1\leq\lambda+\mu<2$. If $\{a_k\}$ satisfy the
following conditions:
\begin{enumerate}
\item $(2-\alpha\lambda)(2-\beta\mu)(b+n-2)a_1\geq (c+n-2)2^{\lambda+\mu+3}a_2$,
\item $k(k+\alpha-\lambda)(k+\beta-\mu)(b+n-k-1)a_k\geq (k+\alpha)(k+\beta)(c+n-k-1)(k+1)a_{k+1}$, for all $2 \leq k\leq n-2$,
\item $(n-1+\alpha-\lambda)(n-1+\beta-\mu)(1+b-c)(n-1)a_{n-1}\geq c(n-1+\alpha)(n-1+\beta)na_n$.
\end{enumerate}
Then $s_n^{(b-1,c)}(f,z)$ is close to convex with respect to starlike
function $z/(1-z^2)$ where $f(z)=z+\displaystyle\sum_{k=2}^{\infty}a_kz^k$, $z\in\mathbb{D}$.
\end{theorem}
\begin{proof}
$s_n^{(b-1,c)}(f,z)=z+\displaystyle\sum_{k=2}^n \frac{B_{n-k}}{B_{n-1}}a_kz^k$ is close
to convex with respect to $z/(1-z^2)$ if $zs_n^{(b-1,c)}(f,z)'$ is typically real function.
Consider
\begin{align*}
zs_n^{(b-1,c)}(f,z)'=z+\sum_{k=2}^n \frac{B_{n-k}}{B_{n-1}}ka_kz^k=b_1 z+\sum_{k=2}^n b_kz^k,
\end{align*}
where $b_1=1$ and $b_k=\frac{B_{n-k}}{B_{n-1}}ka_k$ for $k\geq 2$. Clearly
\begin{align*}
b_{k+1}=\frac{B_{n-k-1}}{B_{n-1}}(k+1)a_{k+1}
\Rightarrow
\left\{
  \begin{array}{ll}
    b_{k+1}=\frac{(c+n-k-1)}{(b+n-k-1)}\frac{(k+1)a_{k+1}}{k a_k}b_k, & \hbox{forall $1\leq k\leq n-2$;} \\
    b_n=\frac{c}{1+b-c}\frac{na_n}{(n-1)a_{n-1}}b_{n-1}, & \hbox{$k=n-1$.}
  \end{array}
\right.
\end{align*}
%For proving our theorem $b_k$ satisfy conditions of Corollary \ref{cor:new-alpha-beta-positive-sum}.
Now,
\begin{align*}
\frac{1}{(2+\alpha)^{\lambda}(2+\beta)^{\mu}}b_1-b_2
&\geq \frac{1}{2^{\lambda+\mu+2}}(2-\alpha\lambda)(2-\beta\mu)b_1-\left(\frac{c+n-2}{b+n-2}\right)\frac{2a_2}{a_1}b_1\\
&=\frac{b_1}{2^{\lambda+\mu+2}}\left[(2-\alpha\lambda)(2-\beta\mu)
- \left(\frac{c+n-2}{b+n-2}\right)\frac{2^{\lambda+\mu+3}a_2}{a_1}\right]\\
&\geq 0.
\end{align*}
Further, for $2\leq k\leq n-2$,
\begin{align*}
&\frac{(k+\alpha)^{\lambda}(k+\beta)^{\mu}}{(k+1+\alpha)^{\lambda}(k+1+\beta)^{\mu}}b_k-b_{k+1}\\
&\geq\left(1-\frac{\lambda}{k+\alpha}\right)\left(1-\frac{\mu}{k+\beta}\right)b_k-
\frac{(c+n-k-1)}{(b+n-k-1)}\frac{(k+1)a_{k+1}}{k a_k}b_k\\
&=\frac{b_k}{(k+\alpha)(k+\beta)}\bigg[(k+\alpha-\lambda)(k+\beta-\mu)\\
&\quad \quad -(k+\alpha)(k+\beta)
\frac{(c+n-k-1)}{(b+n-k-1)}\frac{(k+1)a_{k+1}}{k a_k} \bigg]\geq0.
\end{align*}
For $k=n-1$,
\begin{align*}
& \left(1+\frac{1}{n-1+\alpha}\right)^{-\lambda}\left(1+\frac{1}{n-1+\beta}\right)^{-\mu}b_{n-1}-b_n\\
& \geq \left( 1-\frac{\lambda}{n-1+\alpha}\right)\left(1-\frac{\mu}{n-1+\beta} \right)b_{n-1}-
\frac{c}{(1+b-c)}\frac{na_n}{(n-1)a_{n-1}}b_{n-1},
\end{align*}
which is non-negative. Following the same argument as in Theorem \ref{thm:new-positivity-starlike} ,
$zs_n^{(b-1,c)}(f,z)$ is typically real which completes the proof.
\end{proof}

\begin{remark}
Note that, we have no result for the close-to-convexity of $s_n^{(b-1,c)}(f,z)$
with respect to the starlike function $z/(1-z)^2$ and $z/(1-z+z^2)$. Although
there are not many results in the literature for close-to-convexity with respect
to $z/(1-z+z^2)$, it will be
interesting if one can find the results in this direction.
\end{remark}

\textbf{Acknowledgement:}{ The first author is thankful to the
"Council of Scientific and Industrial Research, India"
(grant code: 09/143(0827)/2013-EMR-1) for financial support to carry out
the above research work.
}


\begin{thebibliography}{99}
\bibitem{ponnusamy-naik-gen-means-2004}M. R. Agrawal, P.G. Howlett, S. K. Lucas, S. Naik\
and\ S. Ponnusamy, Boundedness of generalized
Ces\'aro averaging operators on certain function spaces, J. Comput. Appl. Math.
{\bf 180} (2005), no.~2, 333--344.

\bibitem{ali-saiful-zero-free-2013}R. M. Ali, S. R. Mondal\ and\ V. Ravichandran, Zero-free approximants
to derivatives of prestarlike functions, J. Inequal. Appl. {\bf 2013}, 2013:401, 8 pp.


\bibitem{askey-gasper-biebarbach-conj-AMS-1986} R. Askey\ and G. Gasper,  Inequalities for polynomials, In: The Bieberbach conjecture (A. Baernstein II, D. Drasin, P. Duren, A. Marden, eds.), Math. surveys and monographs (no. 21), Amer. Math. Soc., providence, RI, 1986, 7--32.

\bibitem{brown-dai-wang-2007-ext-viet-ramanujan}G. Brown, F. Dai\ and\ K. Wang,
Extensions of Vietoris's inequalities.
  I, Ramanujan J. {\bf 14} (2007), no.~3, 471--507.

\bibitem{bustoz-1975-AMS}J. Bustoz, Jacobi polynomial sums and
univalent Ces\`aro means, Proc. Amer. Math. Soc. {\bf 50} (1975), 259--264.

\bibitem{dimitrov-extremal-trig-pol-2002}D. K. Dimitrov, Extremal positive trigonometric polynomials.
Approx. Theory, 136--157, {\it DARBA, Sofia}, (2002).

\bibitem{dimitrov-merlo-2002-cons-approx}D. K. Dimitrov\ and\ C. A. Merlo,
Nonnegative trigonometric polynomials, Constr. Approx. {\bf 18} (2002), no.~1, 117--143.

\bibitem{signal-processing-book} D. A. Dumitrescu\, Positive trigonometric polynomials
    and signal processing applications, Signals and Communication Technology. Springer, Dordrecht, (2007).

\bibitem {duren-1983-book} {\sc P.L. Duren}, Univalent Functions,
Springer--Verlag, Berlin, 1983.

\bibitem{Fernandez-2004}J. J. Fern\'andez - Dur\'an,
Circular distributions based on nonnegative trigonometric sums, Biometrics {\bf 60} (2004), no.~2, 499--503.

\bibitem{gasper-1969-nonneg-sum-JMAA} G. Gasper, Nonnegative sums of cosine, ultraspherical and Jacobi polynomials,
 J. Math. Anal. Appl. {\bf 26} (1969), 60--68.

\bibitem{gluchoff-hartman-1998-AMS}A. Gluchoff\ and\ F. Hartmann, Univalent
polynomials and non-negative trigonometric sums, Amer. Math. Monthly {\bf 105} (1998), no.~6, 508--522.

\bibitem {goodman-1983-book} A. W. Goodman, {\it Univalent functions. Vol. I}, Mariner, Tampa, FL, 1983.


\bibitem {koumandos-2007-ext-viet-ramanujan} S. Koumandos, An extension of Vietoris's
inequalities, Ramanujan J. {\bf 14} (2007), no.~1, 1--38.

\bibitem {lewis-1979-convolution-jacobi-SIAM} J. L. Lewis, Applications of a convolution theorem to Jacobi polynomials,
 SIAM J. Math. Anal. {\bf 10} (1979), no.~6, 1110--1120.

\bibitem {saiful-swami-2011-CAMWA} S. R. Mondal\ and\ A. Swaminathan, On the
positivity of certain trigonometric
 sums and their applications, Comput. Math. Appl. {\bf 62} (2011), no.~10, 3871--3883.

\bibitem{nehari}Nehari, Zeev. Conformal mapping. McGraw-Hill Book Co., Inc., New York, Toronto, London, 1952. {\rm viii}+396 pp.

\bibitem{pommerenke-1975-book} C. Pommerenke, {\it Univalent functions},
Vandenhoeck \&\ Ruprecht, G\"ottingen, 1975.

\bibitem{revesz-thesis} Sz. Gy. R\'ev\'esz, Extremal problems for positive definite
functions and polynomials, {\it Thesis for the degree} Doctor of Academy,
pp. 164 (2009), Budapest.

\bibitem{ruscheweyh-1977-prestarlike}  S. Ruscheweyh, Linear operators
between classes of prestarlike
functions, Comment. Math. Helv. {\bf 52} (1977), no.~4, 497--509.

\bibitem{ruscheweyh-1982-book}S. Ruscheweyh, {\it Convolutions in geometric function theory},
S\'eminaire de Math\'ematiques Sup\'erieures, 83, Presses Univ. Montr\'eal, Montreal, QC, 1982.

\bibitem{ruscheweyh-1987-viet-cond-starlike-func}S. Ruscheweyh, Coefficient
conditions for starlike functions,
Glasgow Math. J. {\bf 29} (1987), no.~1, 141--142.

\bibitem{ruscheweyh-1992-geom-cesaro-result-in-math}S. Ruscheweyh, Geometric properties of the Ces\`aro means,
 Results Math. {\bf 22} (1992), no.~3-4, 739--748.

\bibitem{ruscheweyh-salinas-2004-stable-JMAA}S. Ruscheweyh\ and\ L. Salinas, Stable functions and Vietoris' theorem,
J. Math. Anal. Appl. {\bf 291} (2004), no.~2, 596--604.

\bibitem{sangal-swaminathan-starlike-HGF} P.Sangal\ and\
A. Swaminathan, Starlikeness of Gaussian hypergeometric functions using positivity techniques,
Bull. Malays. Math. Sci. Soc.(2016),DOI:10.1007/s40840-016-0420-5.

\bibitem{sangal-swaminathan-positivity-alpha-beta} P. Sangal\ and\
A. Swaminathan, Extension of Vietoris' inequalities for positivity of
trigonometric polynomials, available at, {\tt http://arxiv.org/abs/1705.03759}.

\bibitem{stempak-cesaro-1994-Edinburg}K. Stempak, Ces\`aro averaging operators, Proc. Roy.
Soc. Edinburgh Sect. A {\bf 124} (1994), no.~1, 121--126.

\bibitem{vietoris-1958} L.Vietoris, Über das Vorzeichen gewisser trignometrishcher Summen,
Sitzungsber, Oest. Akad. Wiss. 167 1958,125–-135.



\end{thebibliography}
\end{document}